\documentclass[reqno,12pt]{amsproc}
\usepackage{ucs}
\usepackage[english]{babel}
\usepackage{amssymb,amsmath}
\usepackage{indentfirst}
\usepackage[usenames]{color}
\usepackage{graphicx}
\usepackage[left=2cm,right=2cm,top=1.5cm,bottom=1.5cm,bindingoffset=0cm]{geometry}
\usepackage{amsthm}
\usepackage{amsfonts}
\usepackage{graphicx}
\graphicspath{{pictures/}}
\DeclareGraphicsExtensions{.pdf,.png,.jpg}
\usepackage{wrapfig}

\theoremstyle{plain}  
\newtheorem*{acknowledgements}{Acknowledgements}
\newtheorem{theorem}{\bf Theorem}
\newtheorem{lemma}{\bf Lemma}

\newtheorem{prop}{\bf Proposition}

\newcommand{\g}{\gamma}

\usepackage{xcolor}
\usepackage{hyperref}

\hypersetup{pdfstartview=FitH, linkcolor=linkcolor,urlcolor=urlcolor, colorlinks=true, linkcolor=blue, citecolor=blue}

\begin{document}
\title{Univalence of $T$-symmetric Suffridge type polynomials \\of degree $3T+1$}
\author{Kristina Oganesyan}
\address{Lomonosov Moscow State University, Moscow Center for Fundamental and Applied Mathematics, Centre de Recerca Matem\`atica, Universitat Aut\`onoma de Barcelona}
\email{oganchris@gmail.com}
\thanks{This research was supported by the Russian Science Foundation (project no. 22-21-00545).}
\date{}

\begin{abstract}
We show the univalence of $T$-symmetric Suffridge type polynomials $S_4^{(T)}$ in the unit disk, confirming thereby the conjecture proposed by Dmitrishin, Gray, and Stokolos in their recent paper. The result also implies the quasi-extremality of $S_n^{(T)}$ in the sense of Ruscheweyh.
\end{abstract}
\maketitle

\section{Introduction}
An important family of univalent in $\mathbb{D}$ polynomials, which posseses numerous extremal properties, is that of Suffridge \cite{S} defined by
\begin{align*}
S_{k,N}(z):=z+\sum_{j=2}^N\Big(1-\frac{j-1}{N}\Big)\frac{\sin\frac{\pi kj}{N+1}}{\sin\frac{\pi k}{N+1}}z^j,\quad k=1,2,...,N.
\end{align*}
First of all, one can see that $S_{k,N}$ have the maximal possible modulus of the leading coefficient provided that the first coefficient is $1$. Indeed, a polynomial of the form $z+\sum_{n=2}^Na_nz^n$ with $|a_N|>1/N$ has a zero in $\mathbb{D}$, which means that there is no univalence. Besides, the Suffridge polynomials are quasi-extremal in the sense of Ruscheweyh \cite[Def. 11]{Ru} (this concept is strongly related to the maximal range problem for polynomials, see \cite{CR1,CR2}). The polynomials $S_{1,N}(z)$ are also known to maximize the absolute value of each coefficient of univalent in $\mathbb{D}$ polynomials of degree $N$ with the first coefficient being $1$. Moreover, $S_{1,N}$ minimizes the value $|P(-1)/P(1)|$ among all the univalent in $\mathbb{D}$ polynomials $P$ of degree $N$ with real coefficients and attains in $\overline{\mathbb{D}}$ the maximal possible absolute value of $P$ in the same class under the additional assumption $P(0)=0,\;P'(0)=1$. 

The following class of $T$-symmetric polynomials with fold symmetry was introduced in \cite{DST} as a cantidate to inherit the properties of the Suffridge polynomials: 
\begin{align*}
S_{n}^{(T)}(z):=z+\sum_{j=2}^n\Big(1-\frac{(j-1)T}{1+(n-1)T}\Big)\prod_{k=1}^{j-1}\frac{\sin\frac{\pi(2+T(k-1))}{2+T(n-1)}}{\sin\frac{\pi Tk}{2+T(n-1)}}z^{T(j-1)+1}.
\end{align*}
Note that the leading coefficient of $S_n^{(T)}$ once again is the maximal possible for a univalent in $\mathbb{D}$ polynomial of degree $T(n-1)+1$ with the first coefficient being $1$.
 For a comprehensive surwey on extremal problems for the coefficients of $T$-symmetric univalent in $\mathbb{D}$ functions, we refer the reader to \cite{DGS}.
 
In \cite{DGS}, a number of conjectures concerning the extremal features of $S_n^{(T)}$, similar to those known for $S_{1,N}$, was launched. We prove that the hypotheses a) and b) in \cite[Sec. 2]{DGS} are indeed true for $S_4^{(T)}$. Namely, we show that $S_{4}^{(T)}$ are univalent in $\mathbb{D}$ and quasi-extremal in the sense of Ruscheweyh.

\section{Results and proofs}
\begin{theorem}\label{univ}
The polynomials $S_{4}^{(T)}$ are univalent in the unit disk.
\end{theorem}

First of all, for the sake of convenience we adopt the notation
\begin{align*}
a_T:=\frac{\sin\frac{2\pi}{3T+2}}{\sin\frac{T\pi}{3T+2}},\quad u_T:=Ta_T=T\frac{\sin\frac{2\pi}{3T+2}}{\sin\frac{T\pi}{3T+2}},\quad v_T:=(T+1)a_T=(T+1)\frac{\sin\frac{2\pi}{3T+2}}{\sin\frac{T\pi}{3T+2}},
\end{align*}
so that for $n=4$, the polynomials $S_n^{(T)}$ take the form
\begin{align*}
S_4^{(T)}(z)=z+\frac{2T+1}{3T+1}a_Tz^{T+1}+\frac{T+1}{3T+1}a_Tz^{2T+1}+\frac{1}{3T+1}z^{3T+1}.
\end{align*}

We start with an auxiliary result which bounds $u_T$ and $v_T$ from above. 

\begin{lemma}\label{imp}
For $T\geq 3$, $u_T$ increases and $v_T$ decreases in $T$, moreover, there holds
\begin{align*}
u_T< 2.5,\quad v_T< (0.345)^{-1}.
\end{align*}
\end{lemma}

\begin{proof}
We start with $u_T$. We have
\begin{align}\label{rhs}
&\frac{(3T^2+2T)^2\sin^2\frac{2\pi}{3T+2}}{2\pi\cos\frac{2\pi}{3T+2}\cos\frac{T\pi}{3T+2}}\Big(\frac{1}{u_T}\Big)_T'=T\tan\frac{2\pi}{3T+2}-\tan\frac{T\pi}{3T+2}\Big(-3T+\frac{(3T+2)^2}{2\pi}\tan\frac{2\pi}{3T+2}\Big).
\end{align}
Note that the function $\tan(kx)/k$ is increasing in $k\in(0,\pi/2x)$, since 
\begin{align*}
\Big(\frac{\tan kx}{k}\Big)'_k=\frac{2kx-\sin 2kx}{2k^2\cos^2kx}>0,
\end{align*}
whence the right-hand side of \eqref{rhs} does not exceed
\begin{align*}
T\tan\frac{2\pi}{3T+2}&\Big(1-0.5\Big(-3T+\frac{(3T+2)^2}{2\pi}\tan\frac{2\pi}{3T+2}\Big)\Big)\\
&<0.5T\tan\frac{2\pi}{3T+2}(2-(-3T+(3T+2)))=0.
\end{align*}
Thus, $u_T$ increases and
\begin{align*}
u_T\leq \lim_{T\to\infty}u_T=\frac{4\pi}{3\sqrt{3}}<0.4.
\end{align*}

Turn now to $v_T$. For $T\leq 5$, the monotonicity can be checked by direct computations, so we assume that $T\geq 5$. Note that $\sin\frac{T\pi}{3T+2}$ increases, therefore it suffices to show that $(T+1)\sin\frac{2\pi}{3T+2}$ decreases. 

Note that for $x\leq 2\pi/17$ there holds $\tan x<x(1+x/2\pi)$. Indeed, it is enough to verify that
\begin{align*}
(\tan x)'=\cos^{-2}x\leq (1-x^2)^{-1}<1+\frac{x}{\pi}=\Big(x\Big(1+\frac{x}{2\pi}\Big)\Big)'.
\end{align*}
Hence,
\begin{align*}
\Big((T+1)\sin\frac{2\pi}{3T+2}\Big)'_T=\sin\frac{2\pi}{3T+2}-\frac{6\pi(T+1)}{(3T+2)^2}\cos\frac{2\pi}{3T+2}<0,
\end{align*}
which implies the monotonicity of $v_T$. To conclude the proof, it remains to observe that $v'_{3}\leq (0.345)^{-1}$.
\end{proof}

\begin{proof}[Proof of Theorem \ref{univ}] 

Note that we can assume $T\geq 3$, as for $T=1,2$, the univalence of $S_n^{(T)}$ is already known, since $S_n^{(1)}(z)\equiv S_{1,n}(z),\;S_n^{(2)}(z)\equiv -iS_{n,2n-1}(iz)$. 

As we deal with polynomials, in order to prove the univalence it suffices to show that the image of the unit circle is a simple curve. Assume the contrary, so that there holds $S_4^{(T)}(e^{ix})=S_4^{(T)}(e^{iy})$ for some $x> y,\;x,y\in[0,2\pi)$. Then
\begin{align*}
0=\frac{S_4^{(T)}(e^{ix})-S_4^{(T)}(e^{iy})}{e^{ix}-e^{iy}}
&=1+\frac{(2T+1)\sin (T+1)\g}{(3T+1)\sin\g}a_Tw+\frac{(T+1)\sin (2T+1)\g}{(3T+1)\sin\g}a_Tw^2\\
&+\frac{\sin (3T+1)\g}{(3T+1)\sin\g}w^3=:f_{\g}(w),
\end{align*}
where $\g:=(x-y)/2,\;w:=\exp(iT(x+y)/2)$. This means that for some $\g$, the polynomial $f_{\g}$ has a root of modulus $1$, which in turn yields that either $f_{\g}(\pm 1)=0$ or $f_{\g}$ has a pair of conjugate complex roots of modulus $1$. Note that if $\sin (3T+1)\g=0$, then $\g=\pi k/(3T+1)$ for some $0<k<3T+1$, and due to the fact that $2T+1$ and $3T+1$ are coprime, we have $\sin(2T+1)\g\neq 0$. So, in this case $f_{\g}$ becomes a quadratic polynomial with the absolute value of the product of its roots being 
\begin{align*}
\frac{(3T+1)\sin \g}{v_T|\sin(2T+1)\g|}=\frac{(3T+1)\sin\frac{\pi k}{3T+1}}{v_T|\sin\frac{T\pi k}{3T+1}|}\geq \frac{3T+1}{3T}>1,
\end{align*}
according to Lemma \ref{imp}. Thus, either $f_{\g}(\pm 1)=0$ or $f_{\g}(-(3T+1)\sin\g/\sin (3T+1)\g)=0$ provided that $\sin(3T+1)\g\neq 0$. Consider the three cases.

{\bf Case 1. $f_{\g}(1)=0$.} For some $\g\in (0,\pi)$, we have
\begin{align*}
G_1(\g):&=\Big(\sin \g+\frac{1}{3T+1}\sin (3T+1)\g\Big)\sin\frac{T\pi}{3T+2}\\
&+\Big(\frac{2T+1}{3T+1}\sin (T+1)\g+\frac{T+1}{3T+1}\sin (2T+1)\g\Big)\sin\frac{2\pi}{3T+2}=0.
\end{align*}
Note that $G_1(0)=G_1(\pi)=0$. Thus, if $G_1$ vanishes at some $\g\in(0,\pi)$, there exists a local extremum of $G_1$, where $G_1$ is nonpositive. At a local extremum there must hold
\begin{align*}
G'_1(\g)=2\cos\frac{3T+2}{2}\g\Big(\cos\frac{3T}{2}\g\;\sin\frac{T\pi}{3T+2}+\frac{(2T+1)(T+1)}{3T+1}\cos\frac{T}{2}\g\;\sin\frac{2\pi}{3T+2}\Big)=0.
\end{align*}
Assume first that $\cos((3T+2)\g/2)=0$. In this case $\g=(2k+1)\pi/(3T+2)$ for some $0\leq k<(3T+1)/2$, whence
\begin{align}\label{1}
\frac{3T+1}{3T+2}G_1(\g)=\sin\frac{(2k+1)\pi}{3T+2}\sin\frac{T\pi}{3T+2}+\sin\frac{(T+1)(2k+1)\pi}{3T+2}\sin \frac{2\pi}{3T+2}.
\end{align}
The first term in \eqref{1} is positive, so that $G_1(\g)\leq 0$ only if the second one is negative. In the latter case there holds $2k+1\geq 3$. We have several possibilities.

1. If $\sin \g=\sin(\pi/(3T+2))$, then $2k+1=3T+1,$ which yields that $T$ is even, so that the second term in \eqref{1} is positive. 

2. If $\sin \g=\sin(2\pi/(3T+2))$, then there must hold $\g=3T\pi/(3T+2)$, so that $2k+1=3T$, which implies that $T$ is odd. In the latter occasion we get $G_1(\gamma)=0$. 

3. If $\sin \g=\sin(3\pi/(3T+2))$, then the sine in the second term is up to a sign equal to $\sin(3(T+1)\pi)/(3T+2)=\sin(\pi/(3T+2))$. This implies that the second term in \eqref{1} is less than the first one in absolute value, so that the whole sum is positive.

4. If $\sin \g\geq \sin(4\pi/(3T+2))$, then
\begin{align*}
G_1(\g)\geq \sin\frac{4\pi}{3T+2}\sin\frac{T\pi}{3T+2}-\sin\frac{2\pi}{3T+2}\geq \Big(2\cos\frac{2\pi}{11}\sin\frac{3\pi}{11}-1\Big)\sin\frac{2\pi}{3T+2}>0.
\end{align*}
Combining the arguments above, we see that we have $G_1(\g)\geq 0$ at all the local extrema corresponding to $\cos((3T+2)\g/2)=0$, and $G_1(\g)=0$ is only possible for odd $T$ and $\g=3T\pi/(3T+2)$.

Now assume that
\begin{align*}
\cos\frac{3T}{2}\g\;\sin\frac{T\pi}{3T+2}+\frac{(2T+1)(T+1)}{3T+1}\cos\frac{T}{2}\g\;\sin\frac{2\pi}{3T+2}=0,
\end{align*}
whence
\begin{align*}
a_T=-\frac{(3T+1)\cos\frac{3T}{2}\g}{(2T+1)(T+1)\cos\frac{T}{2}\g}.
\end{align*}
Using this we obtain
\begin{align}\label{tbcont}
(3T+1)&\Big(\sin\frac{T\pi}{3T+2}\Big)^{-1}G_1(\g)=2\sin\frac{3T+2}{2}\g\cos\frac{3T}{2}\g+3T\sin\g\nonumber\\
&+2(T+1)a_T\sin\frac{3T+2}{2}\g\cos\frac{T}{2}\g+Ta_T\sin(T+1)\g\nonumber\\
&=T\Big(-\frac{1}{2T+1}2\sin\frac{3T+2}{2}\g\cos\frac{3T}{2}\g+a_T\sin(T+1)\g+3\sin\g\Big)\nonumber\\
&=T\Big(-\frac{1}{2T+1}\sin(3T+1)\g+a_T\sin(T+1)\g+\Big(3-\frac{1}{2T+1}\Big)\sin\g\Big)\nonumber\\
&>T\Big(-\frac{3T+1}{2T+1}\sin\g+a_T\sin(T+1)\g+(3-0.2)\sin\g\Big)\nonumber\\
&>T(a_T\sin(T+1)\g+1.3\sin\g).
\end{align}
Now we need the following
\begin{lemma}\label{third}
For any $y\geq 3,\;x\in (0,\pi/2)$, there holds
\begin{align*}
3\sin xy>-y\sin x.
\end{align*}
\end{lemma}

\begin{proof}
If $\sin xy\geq 0$, the inequality is trivial. Otherwise, choose $m\in\mathbb{N}$ such that there holds $(2m-1)\pi\leq xy\leq 2m\pi$ and observe that $z:=\min(2m\pi-xy,xy-(2m-1)\pi)\leq xy/3.$ Therefore,
\begin{align*}
3|\sin xy|=3\sin z<3\max\Big(1,\frac{z}{x}\Big)\sin x\leq y\sin x,
\end{align*}
and the claim follows.
\end{proof}

Continuing \eqref{tbcont}, we note that $a_T(T+1)=v_T<3$ by Lemma \ref{imp}, whence by Lemma \ref{third}
\begin{align*}
a_T\sin(T+1)\g+1.3\sin\g\geq \sin\g(-1+1.3)>0.
\end{align*}

Thus, we showed that $G_1(\g)$ is always positive for $\g\in (0,\pi)$ if $T$ is even, and nonnegative otherwise. Moreover, the only point of vanishing of $G_1(\g)$ for even $T$ corresponds to $\g=3T\pi/(3T+2)$. For this value, recalling the definitions of $w$ (which is $1$ under the condition of Case 1) and $\g$, we get
\begin{align*}
x-y=\frac{6\pi T}{3T+2},\;x+y=\frac{4\pi m}{T},
\end{align*}
where $m\in\mathbb{Z}_+$.
 Therefore, as $(x-y)/2<(x+y)/2<2\pi-(x-y)/2$,
\begin{align*}
\frac{3T}{3T+2}<\frac{2m}{T}<2-\frac{3T}{3T+2}
\end{align*}
and
\begin{align*}
T-1<\frac{3T^2}{3T+2}<2m<\frac{(3T+4)T}{3T+2}<T+1,
\end{align*}
which means that we must have $2m=T$. This contradicts the fact that $T$ is odd.

{\bf Case 2. $f_{\g}(-1)=0$.} As above, for some $\g\in (0,\pi)$,
\begin{align*}
G_2(\g):&=\Big(\sin \g-\frac{1}{3T+1}\sin (3T+1)\g\Big)\sin\frac{T\pi}{3T+2}\\
&-\Big(\frac{2T+1}{3T+1}\sin (T+1)\g-\frac{T+1}{3T+1}\sin (2T+1)\g\Big)\sin\frac{2\pi}{3T+2}=0.
\end{align*}
Note that $G_2(0)=G_2(\pi)=0$. As in Case 1, let us estimate $G_1$ at its local extrema. We derive
\begin{align*}
G'_2(\g)=2\sin\frac{3T+2}{2}\g\Big(\sin\frac{3T}{2}\g\;\sin\frac{T\pi}{3T+2}-\frac{(2T+1)(T+1)}{3T+1}\sin\frac{T}{2}\g\;\sin\frac{2\pi}{3T+2}\Big).
\end{align*}
First, assume that $\sin((3T+2)\g/2)=0$. Then $\g=2k\pi/(3T+2)$ for some $0\leq k<(3T+2)/2$, whence
\begin{align}\label{2}
\frac{3T+1}{3T+2}G_2(\g)=\sin\frac{2k\pi}{3T+2}\sin\frac{T\pi}{3T+2}-\sin\frac{(T+1)2k\pi}{3T+2}\sin \frac{2\pi}{3T+2}.
\end{align} 
To have $G_2(\g)<0$, the second term in \eqref{2} must be negative, whence $2k\geq 4$.

1. If $\sin \g=\sin(\pi/(3T+2))$, then $2k=3T+1,$ which yields that $T$ is odd, whence the second term in \eqref{2} is positive.

2. If $\sin \g=\sin(2\pi/(3T+2))$, then there holds $2k=3T$, so that $T$ is even. In the latter occasion we have $G_2(\g)=0$. 

3. If $\sin \g\geq\sin(3\pi/(3T+2))$, then using arguments similar to those in Case 1 we get once again $G_2(\g)>0$.

Thus, we are left to consider $\g$ satisfying
\begin{align*}
\sin\frac{3T}{2}\g\;\sin\frac{T\pi}{3T+2}-\frac{(2T+1)(T+1)}{3T+1}\sin\frac{T}{2}\g\;\sin\frac{2\pi}{3T+2}=0,
\end{align*}
which yields that
\begin{align*}
a_T=\frac{(3T+1)\sin\frac{3T}{2}\g}{(2T+1)(T+1)\sin\frac{T}{2}\g}.
\end{align*}
Using this we derive
\begin{align*}
(3T+1)&\Big(\sin\frac{T\pi}{3T+2}\Big)^{-1}G_2(\g)=-2\cos\frac{3T+2}{2}\g\sin\frac{3T}{2}\g+3T\sin\g\nonumber\\
&+2(T+1)a_T\cos\frac{3T+2}{2}\g\sin\frac{T}{2}\g-Ta_T\sin(T+1)\g\nonumber\\
&=T\Big(\frac{1}{2T+1}2\cos\frac{3T+2}{2}\g\sin\frac{3T}{2}\g-a_T\sin(T+1)\g+3\sin\g\Big)\nonumber\\
&=T\Big(\frac{\sin(3T+1)\g}{2T+1}-a_T\sin(T+1)\g+\Big(3-\frac{1}{2T+1}\Big)\sin\g\Big)=:Th_T(\g).
\end{align*}
Observe that if $T$ is even, then $h_T(\g)=h_T(\pi-\g)$, otherwise, we have, for $\g\leq \pi/2$, by Lemma \ref{third}
\begin{align*}
h(\pi-\g)&=-\frac{\sin(3T+1)\g}{2T+1}+a_T\sin(T+1)\g+\Big(3-\frac{1}{2T+1}\Big)\sin\g\\
&\geq -\frac{3\sin\g}{2}-\frac{v_T\sin\g}{3}+\frac{20\sin\g}{7}>0,
\end{align*}
hence, we can further assume that $\g\in(0,\pi/2]$. 

If $(3T+1)\g\leq \pi-\g$, then by Lemma \ref{imp} we have
\begin{align*}
h_T(\g)\geq \sin\g\Big(\frac{1}{2T+1}-v_T+3-\frac{1}{2T+1}\Big)>0.
\end{align*}
If otherwise $\pi/(3T+2)<\g\leq \pi(3T+1)$, then 
\begin{align*}
(T+1)h_t(\g)\geq -v_T+\frac{6T+2}{2T+1}\cdot\frac{3(T+1)}{3T+2}>-v_T+3>0.
\end{align*}
Now assume that $\g\geq \pi/(2T+1)$, which implies $\sin \g\geq 3/(2T+1)$ and
\begin{align*}
Th_T(\g)\geq -\frac{T}{2T+1}-u_T+\frac{60T}{7(2T+1)}\geq -2.5+\frac{53}{21}>0.
\end{align*}
The only case we are left to consider is that of $\pi/(3T+1)<\g<\pi/(2T+1)$. Denoting then $m:=\pi/\g,\;m\in(2T+1,3T+1),$ and using that $\sin\g\geq 3\g/\pi$ and that $a_T\geq 2.5/T$ by Lemma \ref{imp}, we get 
\begin{align*}
h_T(\g)&\geq -\frac{(3T+1-m)\g}{2T+1}-\frac{2.5}{T}+\frac{6T+2}{2T+1}\cdot\frac{3\g}{\pi}\\
&=\frac{3}{m}\Big(\frac{6T+2}{2T+1}-\frac{\pi}{3}\cdot \frac{3T+1}{2T+1}\Big)-\frac{\pi}{2T+1}-\frac{2.5}{T}.
\end{align*} 
The latter expression decreases in $m$, so it remains to show that it is positive for $\g=\pi/(3T+1)$, which follows from
\begin{align*}
-\frac{2.5}{T}+\frac{6T+2}{2T+1}\cdot\frac{3}{3T+1}=\frac{6T^2-13T-5}{2T(2T+1)(3T+1)}>0.
\end{align*}

Thus, we showed that $G_2(\g)$ is always positive for $\g\in (0,\pi)$ if $T$ is odd, and nonnegative otherwise. Moreover, the only point of vanishing of $G_2(\g)$ for odd $T$ corresponds to $\g=3T\pi/(3T+2)$. However, recalling the definitions of $w$ (which is $-1$ in Case 2) and $\g$, we get

\begin{align*}
x-y=\frac{2\pi(3T+1)}{3T+2},\;x+y=\frac{2\pi (2m+1)}{T},
\end{align*}
where $m\in\mathbb{Z}_+$. 
 Thus,
\begin{align*}
\frac{3T+1}{3T+2}<\frac{2m+1}{T}<2-\frac{3T+1}{3T+2}
\end{align*}
and
\begin{align*}
T-1<\frac{(3T+1)T}{3T+2}<2m+1<\frac{(3T+3)T}{3T+2}<T+1,
\end{align*}
so the unique possibitily is $2m+1=T$, which cannot hold since $T$ is even.

{\bf Case 3. $f_{\g}(-(3T+1)\sin\g/\sin (3T+1)\g)=0$.} We have
\begin{align*}
\tilde{f}(\g):&=\sin^2(3T+1)\g \cdot f_{\g}\Big(-\frac{(3T+1)\sin\g}{\sin(3T+1)\g}\Big)=\sin^2(3T+1)\g-(3T+1)^2\sin^2\g\\
&+v_T\Big(-\frac{2T+1}{T+1}\sin(T+1)\g\sin(3T+1)\g+(3T+1)\sin\g\sin(2T+1)\g\Big).
\end{align*}
For $T=3,4,$ noting that $a_3\leq 0.75$ and $a_4\leq 0.6$, we derive $\tilde{f}(\g)<0$ for $\g\in(0,\pi)$ from the inequalities
\begin{align*}
100\cos x-\cos 10 x-99+0.75(33\cos 3x+7\cos 7x-40\cos 4x)<0
\end{align*}
and
\begin{align*}
169\cos x-\cos 13 x-168+0.6(56\cos 4x+9\cos 9x-65\cos 5x)<0,
\end{align*}
which are true for $x\in(0,\pi)$.

Now we assume that $T\geq 5$. Note that $\tilde{f}(\g)=\tilde{f}(\pi-\g)$, so we can restrict ourselves to $\g\leq \pi/2$. Consider the following four cases.

1. First assume that $\g< \pi/(6T+2)$. Then, letting for simplicity $\xi:=-(3T+1)\sin \g/\sin (3T+1)\g<-1$ and denoting by $\alpha$ the real part of the other two zeros of $f_{\g}$, which according to our assumption are conjugate and have modulus $1$, we get by Vieta's formulas  that
\begin{align*}
\begin{cases}
2\alpha\xi+1=\frac{(2T+1)\sin(T+1)\g}{(T+1)\sin(3T+1)\g}v_T=:a,\\
2\alpha+\xi=-\frac{\sin(2T+1)\g}{\sin(3T+1)\g}v_T=:b.
\end{cases}
\end{align*}
If $\g\geq\pi/(9T+3)$, then $|\xi|\leq \pi/2$ and $\sin(2T+1)/\sin(3T+1)\geq \sin(2\pi/9)/\sin(\pi/3)>0.74$, so by Lemma \ref{imp} $|b|\geq 0.74(T+1)/0.47T>|\xi|$, whence $\alpha<0$. Otherwise, $\g\leq\pi/(9T+3)$, so that $|\xi|\leq 2\pi/3\sqrt{3}<1.21$ and $|b|\geq 2(T+1)(2T+1)/T(3T+1)>4/3>|\xi|$, hence, once again $\alpha<0$. However, $|a|>|b|$, whence there must hold $|2\alpha|>1$, which is impossible since $b-\xi>-1$ by 

\begin{lemma}\label{new}
For any $T\geq 5$, there holds 
\begin{align*}
v_T\sin(2T+1)\g-(3T+1)\sin\g-\sin(3T+1)\g<0
\end{align*}
for all $\g\in(0,\frac{\pi}{6T+2})$.
\end{lemma}

\begin{proof}
Note that
\begin{align*}
\frac{2T+\frac{4}{3}}{3T+1}v_T\leq \frac{10+\frac{4}{3}}{16}v_{5}<\frac{17}{12\cdot 0.736} <1.93,
\end{align*}
whence, for $\beta:=(2T+4/3)/2\g$,
\begin{align*}
g_1'&(\g):=(v_T\sin(2T+4/3)\g-(3T+1)\sin\g-\sin(3T+1)\g)'\\
&<(3T+1)(1.93\cos 2\beta-2\cos^2(1.5\beta))=(3T+1)(-4\cos^3\beta+3.86\cos^2\beta+3\cos\beta-2.93)<0,
\end{align*}
as $\beta\leq \pi/6$, which completes the proof, since our function does not exceed $g_1(\g)$.
\end{proof}

2. Now let $\pi/(6T+2)\leq\g< \pi/(3T+1)$. Since $\g<\pi/10$, there holds $s_{\g}:=\g/\sin \g<1.02$. Denoting $\eta:=(3T+1)\sin\g\geq 1.5$, we obtain due to Lemma \ref{imp}
\begin{align*}
\tilde{f}(\g)\leq 1-\eta^2+v_T(\eta-1)\sin\frac{2s_{\g}\eta}{3}\leq (1-\eta)\Big(1+\eta-\frac{1}{0.34}\sin\frac{2s_{\g}\eta}{3}\Big)=:(1-\eta)g_2(\eta).
\end{align*}
We have
\begin{align*}
g_2'(\eta)=1-\frac{2s_{\g}}{1.02}\cos\frac{2s_{\g}\eta}{3}>1-2\cdot 0.5=0,
\end{align*}
while 
\begin{align*}
g_2\Big((3T+1)\sin\frac{\pi}{6T+2}\Big)\geq 1+\frac{\pi}{2s_{\g}}-v_T\cdot\frac{\sqrt{3}}{2}>1+\frac{\pi}{2.04}-\frac{\sqrt{3}}{0.69}>0,
\end{align*}
which yields that $\tilde{f}(\g)<0$.

3. Consider next $\pi/(3T+1)\g< 3\pi/(6T+2)$. Denoting for simplicity $\delta:=\pi/3,\;\beta:=3(3T+1)\g/\pi$, we observe that $(3T+1)\sin\g\geq \beta\geq 3$ and get by Lemma \ref{imp}

\begin{align*}
\tilde{f}(\g)&\leq 1-(3T+1)\sin\g(\beta- v_T\sin\frac{2\delta\beta}{3})+2u_T\frac{7}{6}\sin \delta\beta\\
&\leq 1-\beta^2+\frac{\beta}{0.34}\sin\frac{2\delta\beta}{3}+\frac{35}{6}\sin \delta\beta=:g_3(\beta).
\end{align*}
Further,
\begin{align*}
g'_3(\beta)=-2\beta+\frac{35}{6}\delta\cos \delta\beta+\frac{1}{0.34}\sin\frac{2\delta\beta}{3}+\frac{\delta\beta}{0.34}\cos\frac{2\delta\beta}{3}<-2\beta+3\sin\frac{2\delta\beta}{3}<0,
\end{align*}
whence to prove that $\tilde{f}(\g)<0$, it remains to show that
\begin{align*}
g_3(3)<-8+9\frac{\sqrt{3}}{2}<0.
\end{align*}

4. Finally, let $\g\geq 3\pi/(6T+2)$. Then $\eta=(3T+1)\sin\g\geq 4.5$ and in light of Lemma \ref{imp} we obtain
\begin{align*}
\tilde{f}(\g)=1-\eta(\eta-v_T)+(2T+1)u_T\leq 1-4.5\Big(4.5-\frac{1}{0.34}\Big)+\frac{7}{3}\cdot\frac{5}{2}<0.
\end{align*}
So, we proved that $\tilde{f}(\g)<0$, which means that $f_{\g}(-(3T+1)\sin\g/\sin(3T+1)\g)< 0$.

Thus, in each of Cases 1--3 we came to contradictions, which yields that $f_{\g}$ does not have a root of modulus $1$. Hence, our assumption was wrong and $S_4^{(T)}$ is injective in the unit circle. This concludes the proof of the theorem.
\end{proof}

With Theorem \ref{univ} in hand, we are in a position to prove 

\begin{prop}\label{q-e} The polynomials $S_{4}^{(T)}$ are quasi-extremal in the sense of Ruscheweyh.
\end{prop}

\begin{proof}
As we have already proved that $S_{4}^{(T)}$ are univalent in the unit disk, it suffices to show that all the zeros of $(S_4^{(T)})'$ are different and lie on the unit circle (see \cite[p. 284]{Ru}). We obtain
\begin{align*}
(S_4^{(T)})'(z)&=(1+z^{T})\Big(1+\Big(\frac{2T+1}{3T+1}v_T-1\Big)z^{T}+z^{2T}\Big)=:(1+z^{T})(1+az^{T}+z^{2T}).
\end{align*}
Note that $|a|<2$, since due to Lemma \ref{imp}
\begin{align*}
-1<a=\frac{2T+1}{3T+1}v_T-1< 3\frac{2T+1}{3T+1}-1<2,
\end{align*}
so that the quadratic polynomial $1+ax+x^2$ has two different complex zeros of modulus $1$. Hence, all the zeros of $(S_4^{(T)})'$ are different and have modulus $1$, as desired.
\end{proof}

\begin{acknowledgements} I am grateful to Alex Stokolos for bringing my attention to the problem and acquainting me with related questions.
\end{acknowledgements}

\end{document}